\documentclass{amsart}
\usepackage{amsmath,amssymb}
\usepackage{enumerate}

\newcommand{\PGL}{\mathop{\mathrm{PGL}}}
\newcommand{\PSL}{\mathop{\mathrm{PSL}}}

\newcommand{\PG}{\mathop{\mathrm{PG}}}

\newcommand{\Sym}{\mathop{\mathrm{Sym}}}

\newcommand{\Cay}{\mathop{\mathrm{Cay}}}
\newcommand{\Irr}{\mathop{\mathrm{Irr}}}
\newcommand{\Aut}{\mathop{\mathrm{Aut}}}

\newcommand{\Gcd}{\mathop{\mathrm{gcd}}}

\newcommand\one{{\bf1}}

\newtheorem{theorem}{Theorem}[section]
\newtheorem{proposition}[theorem]{Proposition}
\newtheorem{lemma}[theorem]{Lemma}

\newtheorem{corollary}[theorem]{Corollary}
\newtheorem{conjecture}[theorem]{Conjecture}
\newtheorem{remark}[theorem]{Remark}
\begin{document}

\title[An Erd\H{o}s-Ko-Rado-type theorem for $\PGL_3(q)$]{An
  Erd\H{o}s-Ko-Rado theorem for the derangement graph of $\PGL_3(q)$
  acting on the projective plane}

\author[K. Meagher]{Karen Meagher}
\address{ Department of Mathematics and Statistics,\newline
University of Regina, 3737 Wascana Parkway, S4S 0A4 Regina SK, Canada}\email{karen.meagher@uregina.ca}

\author[P. Spiga]{Pablo Spiga}
\address{
Dipartimento di Matematica e Applicazioni, University of Milano-Bicocca,\newline
Via Cozzi 53, 20125 Milano, Italy}\email{pablo.spiga@unimib.it}

\thanks{The first author is supported by NSERC.\\ \noindent Address correspondence to Pablo Spiga. (pablo.spiga@unimib.it)}


\begin{abstract}
In this paper we prove an Erd\H{o}s-Ko-Rado-type theorem for intersecting sets of permutations. We show that an intersecting set of maximal size in the projective general linear group $\PGL_3(q)$, in its natural action on the points of the projective line, is either a coset of the stabilizer of a point or a coset of the stabilizer of a line. This gives the first evidence to the veracity of~\cite[Conjecture~$2$]{KaPa}.
\end{abstract}

\maketitle

\section{General results}\label{generalresults}

The Erd\H{o}s-Ko-Rado theorem~\cite{ErKoRa} determines the cardinality
and describes the structure of a set of maximal size of intersecting
$k$-subsets from $\{1,\ldots,n\}$. Specifically, the theorem says that
provided that $n > 2k$, a set of maximal size of intersecting
$k$-subsets from $\{1,\dots, n\}$ has cardinality ${n-1\choose k-1}$
and is the set of all $k$-subsets that contain a common fixed
element. Analogous results hold for many other objects other than
sets, and in this paper we are concerned with an extension of the
Erd\H{o}s-Ko-Rado theorem to permutation groups.

Let $G$ be a permutation group on $\Omega$. A subset $S$ of $G$ is
said to be \emph{intersecting} if for every $g,h\in S$ the permutation
$gh^{-1}$ fixes some point of $\Omega$ (note that this implies that
$\alpha^g = \alpha^h$, for some $\alpha \in \Omega$). As with the
Erd\H{o}s-Ko-Rado theorem, we are interested in finding the
cardinality of an intersecting set of maximal size in $G$ and
classifying the sets that attain this bound.  This problem can be
formulated in a graph-theoretic terminology.  We denote by $\Gamma_G$
the \emph{derangement graph} of $G$; the vertices of this graph are
the elements of $G$ and the edges are the pairs $\{g,h\}$ such that
$gh^{-1}$ is a \emph{derangement}, that is, $gh^{-1}$ fixes no
point. An intersecting set of $G$ is simply an \emph{independent
  set} or a \emph{coclique} of $\Gamma_G$.

The natural extension of the Erd\H{o}s-Ko-Rado theorem for the
symmetric group $\Sym(n)$ was proved in~\cite{CaKu} and~\cite{LaMa}.
These papers, using different methods, showed that every independent
set of $\Gamma_{\Sym(n)}$ has size at most $(n-1)!$. They both further
showed that the only independent sets meeting this bound are the
cosets of the stabilizer of a point. The same result was also proved
in~\cite{GoMe} using the character theory of $\Sym(n)$.

Recently there have been many papers questioning if the natural
extension of the Erd\H{o}s-Ko-Rado theorem holds for specific
permutation groups $G$ (see~\cite{AhMe, MR2302532,
  KaPa,MR2419214}). This means asking if the largest independent sets
in the derangement graph $\Gamma_G$ are the cosets in $G$ of the
stabilizer of a point. Typically, for a general permutation group, the
derangement graph may have independent sets of size larger than the
size of the stabilizer of a point, let alone hope that every such
independent set is the coset of the stabilizer of a point.  However, a
behavior very similar to $\Sym(n)$ is offered by $\PGL_2(q)$ in its
natural action on the projective line~\cite[Theorem~$1$]{KaPa}; the
independent sets of maximal size in the derangement graph of
$\PGL_2(q)$ are exactly the cosets of the stabilizer of a point. It is
not hard to see that for a generic projective linear group there
are maximum independent sets in the derangement graph that are not the
stabilizer of a point.  This lead to the following conjecture:

\begin{conjecture}[ {\cite[Conjecture~$2$]{KaPa}} ] The independent sets
  of maximal size in the derangement graph of $\PGL_{n+1}(q)$ acting
  on the points of the projective space $\PG_n(q)$ are exactly the
  cosets of the stabilizer of a point and the cosets of the stabilizer
  of a hyperplane.
\end{conjecture}

This conjecture is appealing in that, like the other versions of the
Erd\H{o}s-Ko-Rado theorem, it characterizes the independent
sets of maximal size in the derangement graph of a group. However, in this case there are
two distinct families of independent sets of maximal size -- this is similar to the
situation for $k$-subsets from an $n$-set in which any two subsets
have at least $t$ elements in common (see~\cite{MR1429238} for more details).

The main result of this paper gives the first important contribution
towards a proof of this conjecture.
\begin{theorem}\label{main:thrm}
  The independent sets of maximal size in the derangement graph of
  $\PGL_{3}(q)$ acting on the points of the projective plane
  $\PG_2(q)$ are exactly the cosets of the stabilizer of a point and
  the cosets of the stabilizer of a line.
\end{theorem}

In particular, Theorem~\ref{main:thrm}
settles~\cite[Conjecture~$2$]{KaPa} for $n=2$. (The case $n=1$ was
already settled in~\cite[Theorem~$1$]{KaPa}.)  Our proof uses the
method developed in~\cite{GoMe} and hence we make use of some
information on the character theory of $\PGL_3(q)$. Some of our
arguments in the proof of Theorem~\ref{main:thrm} work for any $n\geq
2$, but at some critical juncture we either need some geometric
properties of $\PG_2(q)$ or some algebraic properties of the
irreducible characters of $\PGL_3(q)$. Therefore, we have
decided to present all of our results uniformly with $n=2$. However,
we hope that some of the ideas in this paper might help to give a
proof of~\cite[Conjecture~$2$]{KaPa} for every $n$.

To some extent, the proof of Theorem~\ref{main:thrm} is more involved than the proof of the Erd\H{o}s-Ko-Rado theorem for $\PGL_2(q)$~\cite{KaPa} and for $\Sym(n)$~\cite{GoMe}. In our opinion, this increase of complexity is natural and typical of when the independent sets of maximal size are not  all of the same type.

\section{Notation}\label{sub:1}

We let $q$ be a power of a prime and we  denote by $\mathtt{GF}(q)$ a field of size $q$, and by $\mathtt{GF}(q)^3$
the $3$-dimensional vector space over $\mathtt{GF}(q)$ of row
vectors with basis $(\varepsilon_1,\varepsilon_2,\varepsilon_3)$. The (desarguesian) projective plane $\PG_2(q)$ is the pair $(\mathcal{P},\mathcal{L})$, where the elements of $\mathcal{P}$, respectively $\mathcal{L}$, are the
$1$-dimensional, respectively $2$-dimensional, subspaces of
$\mathtt{GF}(q)^3$. The elements of $\mathcal{P}$ are called points and the elements of $\mathcal{L}$ are called lines of the plane $\PG_2(q)$. For denoting the points we will use greek letters, and for the lines roman letters.

Given two distinct points $\alpha$ and $\alpha'$,
we denote by $\alpha\vee \alpha'$ the line spanned by $\alpha$ and
$\alpha'$, that is, the line of $\PG_2(q)$ containing both
$\alpha$ and $\alpha'$. Moreover, given two distinct lines $\ell$ and
$\ell'$, we denote by $\ell\wedge \ell'$ the point of the intersection
of $\ell$ and $\ell'$.

We denote by $G$ the permutation group $\PGL_3(q)$ in its action on
$\mathcal{P}$, and by $\mathcal{D}$ the set of derangements of $G$.  As is
standard, the subgroup of $G$ that fixes the point $\alpha$ is denoted
by $G_\alpha$ and, similarly, the subgroup that fixes the line $\ell$
is denoted by $G_\ell$.

As usual, $\mathbb{C}[G]$ is the group algebra of $G$ over the complex numbers
$\mathbb{C}$. We  only need the vector space structure on
$\mathbb{C}[G]$: a basis for $\mathbb{C}[G]$ is indexed by the group
elements $g\in G$. 
Given a subset $S$ of $G$, we denote by
$\chi_S\in \mathbb{C}[G]$ the characteristic vector of $S$, that is,
$(\chi_S)_g=1$ if $g\in S$, and $(\chi_S)_g=0$ otherwise. The all ones
vector  is denoted by $\one$ (the length of the vector will be
clear by context).

There exists a natural duality~\cite{Cameron} between the points and the
lines of $\PG_2(q)$ and  this duality is preserved by
$G=\mathrm{PGL}_3(q)$. Hence, for each $g\in G$, the number of
elements of $\mathcal{P}$ fixed by $g$ coincides with the number of
elements of $\mathcal{L}$ fixed by $g$. In particular, we have the
equality
\begin{equation}\label{equation:9}
\sum_{\alpha\in \mathcal{P}}\chi_{G_\alpha}=\sum_{\ell\in \mathcal{L}}\chi_{G_\ell}.
\end{equation}

\section{Proof of Theorem~\ref{main:thrm}}\label{sec:eigenvalues}

To prove Theorem~\ref{main:thrm}, we will show that the characteristic
vector for every independent set of maximal size in the derangement graph of $G$ is a
linear combination of a specific set of vectors. Then we will show
that the only such possible linear combinations are the characteristic
vectors for either the cosets of the stabilizer of a point or the
cosets of the stabilizer of a line.

Let $A$ be the $\{0,1\}$-matrix where the rows are indexed by the
elements of $G$, the columns are indexed by the ordered pairs of
points from $\mathcal{P}$ and $A_{g,(\alpha,\beta)}=1$ if and only
if $\alpha^g=\beta$. In particular, $A$ has $|G|=q^3(q^3-1)(q^2-1)$ rows and
$|\mathcal{P}|^2=(q^2+q+1)^2$ columns.

We fix a particular ordering of the rows of $A$ so that the first rows
are labeled by the derangements $\mathcal{D}$ of $G$, and the
remaining rows are labeled by the elements of
$G\setminus\mathcal{D}$.  With this ordering, we get
that $A$ is the following block matrix
\[
A=
\begin{pmatrix}
M\\
B\\
\end{pmatrix}.
\] 
In particular, the rows of the submatrix $M$ are labeled by elements
of $\mathcal{D}$ and the columns of $M$ are labeled by pairs of
elements of $\mathcal{P}$.

Since the columns of $A$ have coordinates indexed by the elements of
$G$, we can view each column of $A$ as an element of $\mathbb{C}[G]$.
Similarly, the rows of $A$ can be viewed as characteristic vectors in
a suitable vector space.

Define $V$ to be the $\mathbb{C}$-vector space whose
basis consists of all $e_{\alpha\beta}$, where $(\alpha,\beta)$ is an
ordered pair of elements of $\mathcal{P}$. Propositions~\ref{rank} and
\ref{lemma:4} will give the right kernels of $A$ and $M$,
respectively. To describe these kernels, we define three families of
vectors from $V$. Two of these families are indexed by the points
$\alpha\in \mathcal{P}$ and the third by the lines $\ell \in
\mathcal{L}$:
\begin{align}\label{eq:ealpha}
e_\alpha^1=\sum_{\beta\in\mathcal{P}}e_{\alpha\beta}, \qquad 
e_{\alpha}^2=\sum_{\beta\in\mathcal{P}}e_{\beta\alpha}, \qquad 
e_{\ell}=\sum_{\beta,\beta'\in\ell}e_{\beta\beta'}.
\end{align}

We now state three pivotal properties of the matrices $A$ and $M$; we
postpone the technical proofs of all three properties to
Sections~\ref{Proof} and~\ref{sec:Proof2}.

\begin{proposition}\label{prop:matrixA}
  If $S$ is an independent set of maximal size of $\Gamma_G$, then
  $\chi_S$ is a linear combination of the columns of $A$.
\end{proposition}

\begin{proposition}\label{rank}
  The matrix $A$ has rank $(|\mathcal{P}|-1)^2+1$.  Moreover, given
  $\bar{\alpha}\in\mathcal{P}$, the subspace
\[
\langle e_{\bar{\alpha}}^1-e_{\alpha}^1,e_{\bar{\alpha}}^2-e_{\alpha}^2\mid \alpha\in\mathcal{P}\rangle
\]
of $V$ is the right kernel of $A$.
\end{proposition}

\begin{proposition}\label{lemma:4}
Given $\bar{\alpha}\in\mathcal{P}$ and $\bar{\ell}\in\mathcal{L}$, the
  subspace
\[
\langle e_{\alpha\alpha},e_{\bar{\alpha}}^1-e_{\alpha}^1,e_{\bar{\alpha}}^2-e_{\alpha}^2,e_{\bar{\ell}}-e_\ell\mid \alpha \in\mathcal{P},\ell\in\mathcal{L}\rangle
\]
of $V$ is the right kernel of  $M$. 
\end{proposition}

Before proving Theorem~\ref{main:thrm} (using these three yet unproven
propositions), we need to show that if the columns of the matrix $B$
are arranged so that the first $q^2+q+1$ columns correspond to the
ordered pairs of the form $(\alpha, \alpha)$, then there is an
arrangement of the rows such that the top left corner of $B$ forms a
$(q^2+q+1) \times (q^2+q+1)$ identity matrix. To prove this, it is
enough to show that for every $\alpha \in \mathcal{P}$ there is a
permutation in $G$ that has $\alpha$ as its only fixed point. (To
simplify some of the computations in the proof of
Theorem~\ref{main:thrm}, we actually prove something slightly
stronger.)

\begin{lemma}\label{lemma:5}
  For every $\alpha\in\mathcal{P}$ and for every $\ell\in\mathcal{L}$,
  there exists $g\in G$ with $g$ fixing only the element $\alpha$ of
  $\mathcal{P}$ and only the element $\ell$ of $\mathcal{L}$.
\end{lemma}
\begin{proof}
  Observe that $G$ acts transitively on the sets $\{(\alpha,\ell)\mid
  \alpha\in \mathcal{P},\ell\in \mathcal{L},\alpha\in \ell\}$ (this is
  the set of \textrm{flags} of $\PG_2(q)$) and $\{(\alpha,\ell)\mid
  \alpha\in \mathcal{P},\ell\in\mathcal{L},\alpha\notin \ell\}$ (this
  is the set of \textrm{anti-flags} of $\PG_2(q)$). In particular,
  replacing $\alpha$ and $\ell$ by $\alpha^h$ and $\ell^h$ for some
  $h\in G$, we may assume that either $\alpha=\langle
  \varepsilon_1\rangle$ and $\ell=\langle
  \varepsilon_1,\varepsilon_2\rangle$, or $\alpha=\langle
  \varepsilon_1\rangle$ and $\ell=\langle
  \varepsilon_2,\varepsilon_3\rangle$. In the first case,
\[
g=\left[
\begin{array}{ccc}
1&0&0\\
1&1&0\\
0&1&1\\
\end{array}
\right]\in G
\]
fixes only the element $\alpha$ of $\mathcal{P}$ and only the element $\ell$ of $\mathcal{L}$. In the second case, if 
\[\left(
\begin{array}{cc}
a&b\\
c&d
\end{array}\right)
\]
is a $2\times 2$-matrix with coefficients in $\mathtt{GF}(q)$ and with irreducible characteristic polynomial, then
\[
g=\left[
\begin{array}{ccc}
1&0&0\\
0&a&b\\
0&c&d\\
\end{array}
\right]\in G\]
fixes only  the element $\alpha$ of $\mathcal{P}$ and only the element $\ell$ of $\mathcal{L}$.
\end{proof}

Now we have all the tools needed to prove the main result of this paper.

\begin{proof}[Proof of Theorem~$\ref{main:thrm}$]
  Let $S$ be an independent set of maximal size of $\Gamma_G$. We aim to
  prove that $S$ is a coset of the stabilizer of a point or a line
  of $\PG_2(q)$. Up to multiplication of $S$ by a suitable
  element of $G$, we may assume that the identity of $G$ is in $S$. In
  particular, we have to prove that $S$ is the stabilizer of a point
  or a line of $\PG_2(q)$.

  From Proposition~\ref{prop:matrixA}, the characteristic vector
  $\chi_S$ of $S$ is a linear combination of the columns of $A$, and
  hence, for some vector $x$, we have
\[
\chi_S=Ax=\left(
\begin{array}{cc}
M\\
B\\
\end{array}
\right)
x=
\left(
\begin{array}{c}
Mx\\
Bx\\
\end{array}
\right).
\]
As the identity of $G$ is in $S$, there are no derangements in $S$. Hence, by our choice of the ordering of the rows of $A$, we get
\[
\chi_S=
\left(
\begin{array}{c}
0\\
t
\end{array}
\right)
\]
and thus $Mx=0$ and $Bx=t$. Hence $x$ lies in the right kernel of $M$. Therefore, by Proposition~\ref{lemma:4},
given a fixed point $\bar{\alpha}$ and a fixed line $\bar{\ell}$, we
have
\[
x=\sum_{\alpha\in\mathcal{P}}c_\alpha e_{\alpha\alpha}+\sum_{\alpha\in\mathcal{P}\setminus\{\bar{\alpha}\}}c_\alpha^1(e_{\bar{\alpha}}^1-e_{\alpha}^1)+
\sum_{\alpha\in\mathcal{P}\setminus\{ \bar{\alpha}\}}c_\alpha^2(e_{\bar{\alpha}}^2-e_{\alpha}^2)+
\sum_{\ell\in\mathcal{L}\setminus\{\bar{\ell}\}}c_\ell(e_{\bar{\ell}}-e_{\ell}),
\]
for some scalars $c_\alpha,c_\alpha^1,c_\alpha^2,c_\ell\in \mathbb{C}$.

From Proposition~\ref{rank}, the vectors
$e_{\bar{\alpha}}^1-e_\alpha^1$ and $e_{\bar{\alpha}}^2-e_\alpha^2$ are
in the right kernel of $A$, and hence
$B(e_{\bar{\alpha}}^1-e_\alpha^1)=B(e_{\bar{\alpha}}^2-e_\alpha^2)=0$. In
particular, 
\[
t = Bx = B \left( \sum_{\alpha\in\mathcal{P}}c_\alpha e_{\alpha\alpha} + \sum_{\ell\in\mathcal{L}\setminus\{\bar{\ell}\}}c_\ell(e_{\bar{\ell}}-e_{\ell}) \right);
\]
and hence we may assume that $c_\alpha^1=c_\alpha^2=0$, for every $\alpha\in\mathcal{P}$.

For $\alpha\in \mathcal{P}$, we have $Be_{\alpha\alpha}=\chi_{G_\alpha}$, where $\chi_{G_\alpha}$ is the
characteristic vector of the stabilizer $G_\alpha$ of $\alpha$. Moreover, for
$\ell\in \mathcal{L}$ and $g$ an element of $G$ that is not a
derangement, we have
\begin{equation}\label{equation:13}
(Be_{\ell})_g=\sum_{\alpha,\beta\in \ell}B_{g,(\alpha,\beta)}=\sum_{\alpha\in \ell, \alpha^g\in \ell}1=|\{\alpha\in\ell\mid \alpha^g\in \ell\}|=
\begin{cases}
q+1&\textrm{if }g\in G_\ell,\\
1&\textrm{if }g\notin G_\ell.
\end{cases}
\end{equation}
(For the last equality observe that $\ell\wedge\ell^g$ is a point of $\PG_2(q)$ when $\ell\neq \ell^g$.)
Thus $Be_{\ell}= q \chi_{G_{\ell}} + \one$, and it follows
from~\eqref{equation:13} that $B(e_{\bar{\ell}}-e_{\ell})=
q(\chi_{G_{\bar{\ell}}}-\chi_{G_\ell})$.
Putting these two facts together, we get that
\begin{equation}\label{equation:6}
t=Bx=\sum_{\alpha\in\mathcal{P}}c_\alpha \chi_{G_\alpha}+q\sum_{\ell\in\mathcal{L}\setminus\{\bar{\ell}\}}c_\ell(
\chi_{G_{\bar{\ell}}}-\chi_{G_\ell}).
\end{equation}
As the identity of $G$ is in $S$, the coordinate of $t$ corresponding to the
identity of $G$ is $1$ and hence, using the formula
in~\eqref{equation:6} for $t$, we get
\begin{equation}\label{equation:5}
\sum_{\alpha\in\mathcal{P}}c_\alpha=1.
\end{equation}

Applying Lemma~\ref{lemma:5} to $\alpha\in \mathcal{P}$ and to the
line $\bar{\ell}$, we get $g\in G$ that fixes only the point $\alpha$ and
only the line $\bar{\ell}$. As $t$ is a $\{0,1\}$-vector, the coordinate of $t$ corresponding
to $g$ is either $0$ or $1$; by taking the coordinate
corresponding to $g$ on the right-hand side of~\eqref{equation:6}, we
get
\begin{align}\label{equation:7}
c_\alpha+q\sum_{\ell\in\mathcal{L}\setminus\{\bar{\ell}\}}c_\ell \in \{0,1\},
\end{align}
for every $\alpha\in\mathcal{P}$. Applying this argument to $\alpha\in
\mathcal{P}$ and to $\ell\in\mathcal{L}\setminus\{\bar{\ell}\}$, we
obtain
\begin{align}\label{equation:8}
c_\alpha-qc_\ell \in  \{0,1\},
\end{align}
for every $\alpha\in\mathcal{P}$ and every $\ell\in\mathcal{L}\setminus\{\bar{\ell}\}$.

Write $c=\sum_{\ell\in\mathcal{L}\setminus\{\bar{\ell}\}}c_\ell$. From~\eqref{equation:7},
we have $c_\alpha=-qc$ or $c_\alpha=-qc+1$, for every
$\alpha\in\mathcal{P}$. Define the sets
\[
\mathcal{P}_{-qc}=\{\alpha\in\mathcal{P}\mid c_\alpha=-qc\} \quad\textrm{and}\quad
\mathcal{P}_{-qc+1}=\{\alpha\in\mathcal{P}\mid -qc+1\}. 
\]
We will first show that if these sets are both non-empty then $S$ is
the stabilizer of a point.  Next we will show if the first set is
empty then $S$ is the stabilizer of $\bar{\ell}$, and if the second is
empty then $S$ is the stabilizer of some line in $\mathcal{L}
\setminus \{\bar{\ell}\}$.

Suppose that both $\mathcal{P}_{-qc}$ and $\mathcal{P}_{-qc+1}$ are
non-empty. Let $\alpha\in \mathcal{P}_{-qc+1}$ and let $\beta \in
\mathcal{P}_{-qc}$.  Applying~\eqref{equation:8} to $\alpha$
and $\beta$, we get 
\[
-qc+1-qc_\ell\in \{0,1\} \quad\textrm{and}\quad -qc-qc_\ell\in \{0,1\},
\]
for each $\ell\in\mathcal{L}\setminus\{\bar{\ell}\}$.  If
$-qc+1-qc_\ell=0$, for some $\ell\in\mathcal{L}\setminus\{ \bar{\ell}\}$, then
$-qc-qc_\ell=-1$, but this is a contradiction.
Thus $-qc+1-qc_\ell=1$ (and hence $c_\ell=-c$), for every $\ell\in\mathcal{L}\setminus\{\bar{\ell}\}$. 
Since $c=\sum_{\ell\in\mathcal{L}\setminus\{\bar{\ell}\}}c_\ell$, we obtain 
\[
c=\sum_{\ell\in\mathcal{L}\setminus\{ \bar{\ell}\}}(-c)=-c(q^2+q).
\]
This implies that $c=0$, and hence $c_\ell=0$, for each
$\ell\in\mathcal{L}\setminus\{\bar{\ell}\}$. This shows that $t=\sum_{\alpha\in \mathcal{P}}c_\alpha
\chi_{G_\alpha}$. Now,~\eqref{equation:7} gives $c_\alpha\in \{0,1\}$, for every $\alpha\in \mathcal{P}$,
and hence~\eqref{equation:5} implies that there exists a unique $\alpha'\in\mathcal{P}$
with $c_{\alpha'}=1$ and all other scalars are zero. Thus
$t=\chi_{G_{\alpha'}}$ and $S$ is the stabilizer of the point $\alpha'$. 

Next we assume that $\mathcal{P}_{-qc+1}=\emptyset$. Thus $\mathcal{P}=\mathcal{P}_{-qc}$ and $c_\alpha=-qc$,
for each $\alpha\in\mathcal{P}$. In particular,~\eqref{equation:5} gives
$-qc(q^2+q+1)=1$, that is, $c=-1/(q(q^2+q+1))$. Moreover,~\eqref{equation:8}
gives $-qc-qc_\ell\in \{0,1\}$, that is,
\[
c_\ell \in \{ -c, -c-1/q\},
\] 
for each $\ell\in\mathcal{L}\setminus\{\bar{\ell}\}$.
Let $a$ be the number of lines $\ell\in\mathcal{L}\setminus\{ \bar{\ell}\}$ with $c_\ell=-c$ and let $b$ be the
number of lines $\ell\in\mathcal{L}\setminus\{\bar{\ell}\}$ with $c_\ell=-c-1/q$. As $|\mathcal{L}\setminus\{\bar{\ell}\}|=q^2+q$, we have $a+b=q^2+q$. Moreover, by the definition of the parameter $c$,
we have
\[a(-c)+b(-c-1/q)=c.
\]
Putting these together we have
$-c(q^2+q)-b/q=c$ which implies $b=-cq(q^2+q+1)=1$. In particular,
there exists a unique $\ell'\in\mathcal{L}\setminus\{\bar{\ell}\}$ with $c_{\ell'}=-c-1/q$ and all other
lines $\ell\in\mathcal{L}\setminus\{\bar{\ell}\}$ have $c_\ell=-c$. 
Therefore, using~\eqref{equation:9}, we get
\begin{align*}
t&= -cq\sum_{\alpha\in \mathcal{P}}\chi_{G_\alpha}+q(-c)\sum_{\ell\in \mathcal{L}\setminus\{\bar{\ell}\}}(\chi_{G_{\bar{\ell}}}-\chi_{G_\ell})+   
   q \left(\frac{-1}{q}\right) (\chi_{G_{\bar{\ell}}}-\chi_{G_{\ell'}})\\
&=-cq\left(\sum_{\alpha\in \mathcal{P}}\chi_{G_\alpha}-\sum_{\ell\in \mathcal{L}}\chi_{G_\ell}\right)  
    - qc\sum_{\ell\in\mathcal{L}}\chi_{G_{\bar{\ell}}} - (\chi_{G_{\bar{\ell}}} - \chi_{G_{\ell'}})\\
&=-  cq(q^2+q+1)\chi_{G_{\bar{\ell}}} - (\chi_{G_{\bar{\ell}}} - \chi_{G_{\ell'}})=\chi_{G_{\bar{\ell}}}-(\chi_{G_{\bar{\ell}}}-\chi_{G_{\ell'}}) =  \chi_{G_{\ell'}}.
\end{align*}
In this case $S$ is the stabilizer of the line $\ell'$.

Finally, suppose that $\mathcal{P}_{-qc}=\emptyset$. Thus $\mathcal{P}=\mathcal{P}_{-qc+1}$ and
$c_\alpha=-qc+1$, for each $\alpha\in\mathcal{P}$. In particular,~\eqref{equation:5}
gives $(-qc+1)(q^2+q+1)=1$, that is, $c=(q+1)/(q^2+q+1)$.
Furthermore, ~\eqref{equation:8} gives $-qc+1-qc_\ell\in \{0,1\}$, that is,
\[
c_\ell \in \{ -c, -c+1/q\},
\]
for each $\ell\in\mathcal{L}\setminus\{\bar{\ell}\}$.
Again, let $a$ be the number of lines $\ell\in\mathcal{L}\setminus\{\ell\}$ with
$c_\ell=-c+1/q$ and let $b$ be the number of lines $\ell\in\mathcal{L}\setminus\{\ell\}$ with
$c_\ell=-c$. As in the previous case, we have the equations
\[
a+b=q^2+q, \qquad  a(-c+1/q)+b(-c)=c,
\]
and hence $-c(q^2+q)+a/q=c$. Thus
$a=cq(q^2+q+1)=q^2+q$ and $b=0$. Therefore, using~\eqref{equation:9}, we get
\begin{align*}
t&=(-cq+1)\sum_{\alpha\in \mathcal{P}}\chi_{G_\alpha}+q(-c+1/q)\sum_{\ell\in \mathcal{L}\setminus\{\bar{\ell}\}}(\chi_{G_{\bar{\ell}}}-\chi_{G_\ell})\\
&=(-cq+1)\left(\sum_{\alpha\in \mathcal{P}}\chi_{G_\alpha}-\sum_{\ell\in \mathcal{L}}\chi_{G_\ell}\right)+(-cq+1)\sum_{\ell\in\mathcal{L}}\chi_{G_{\bar{\ell}}}\\
&=(-cq+1)(q^2+q+1)\chi_{G_{\bar{\ell}}} =\chi_{G_{\bar{\ell}}}.
\end{align*}
In this case $S$ is the stabilizer of the line $\bar{\ell}$.
\end{proof}

\section{Proof of Propositions~\ref{prop:matrixA} and~\ref{rank}}\label{Proof}

For a moment, we leave part of the notation that we have set so far to
consider general groups.  Let $G$ be a permutation group on $\Omega$
and let $\Gamma_G$ be its derangement graph. Since the right regular
representation of $G$ is a subgroup of the automorphism group $\Aut(\Gamma_G)$ of $\Gamma_G$, we see that
$\Gamma_G$ is a \textit{Cayley graph}. Namely, if $\mathcal{D}$ is the set of
derangements of $G$, then $\Gamma_G$ is the Cayley graph on $G$ with
connection set $\mathcal{D}$, i.e. $\Gamma_G=\Cay(G,\mathcal{D})$.
Clearly, $\mathcal{D}$ is a union of $G$-conjugacy classes, that is,
$\Gamma_G$ is a \emph{normal} Cayley graph.

As usual, we simply say that the complex number $\xi$ is an
\emph{eigenvalue} of the graph $\Gamma$ if $\xi$ is an eigenvalue of
the adjacency matrix of $\Gamma$.  We use $\Irr(G)$ to denote the
\emph{irreducible complex characters} of the group $G$ and given
$\chi\in\Irr(G)$ and a subset $S$ of $G$ we write $\chi(S)$ for
$\sum_{s\in S}\chi(s)$.  In the following lemma we recall that the
eigenvalues of a normal Cayley graph on $G$ are determined by
the irreducible complex characters of $G$, see~\cite{Ba}.

\begin{lemma}\label{eigenvalues}
  Let $G$ be a permutation group on $\Omega$ and let $\mathcal{D}$ be
  the set of derangements of $G$. The spectrum of the graph $\Gamma_G$
  is $\{\chi(\mathcal{D})/\chi(1)\mid \chi\in \Irr(G)\}$. Also, if
  $\tau$ is an eigenvalue of $\Gamma_G$ and $\chi_1,\ldots,\chi_s$ are
  the irreducible characters of $G$ such that
  $\tau=\chi_i(\mathcal{D})/\chi_i(1)$, then the dimension of the
  $\tau$-eigenspace of $\Gamma_G$ is $\sum_{i=1}^s\chi_i(1)^2.$
\end{lemma}

The next result is the well-known \emph{ratio-bound} for independent
sets in a graph; for a proof see (for example)~\cite[Lemma~$3$]{KaPa}.

\begin{lemma}\label{lemma:7}
Let $G$ be a permutation group, let $\tau$ be the minimum eigenvalue of
$\Gamma_G$, let $d$ be the valency of $\Gamma_G$ and let $S$ be an independent set of $\Gamma_G$.  Then
\[
|S|\leq |G|/(1-d/\tau).
\] 
If the equality is met, then $\chi_S-\frac{|S|}{|G|}\chi_{G}$ is an
eigenvector of $\Gamma_G$ with eigenvalue $\tau$.
\end{lemma}

We are now ready to return to only considering the group
$G=\PGL_{3}(q)$.  Let $\pi$ be the permutation character of $G$. As
$G$ is $2$-transitive, we have $\pi=1+\chi_0$, where $1$ is the
principal character of $G$ and $\chi_0$ is an irreducible character of
degree $q^2+q$.

\begin{lemma}\label{eigenvaluesPGL3} Let $\Gamma_G$ be the derangement
  graph of $G = \PGL_3(q)$. Then the following hold.
\begin{enumerate}
\item The largest eigenvalue of $\Gamma_G$ is $|\mathcal{D}|=(q^2-1)^2q^4/3$.
\item The minimum eigenvalue of $\Gamma_G$ is $\tau = -(q-1)(q^2-1)q^3/3$ and, provided that $q>2$, the eigenvalue $\tau$ has multiplicity $(\Gcd(3,q-1) - 1) + (q^2+q)^2$.
\item For $\chi \in \Irr(G)$,
\[
\chi(\mathcal{D})/\chi(1)=\tau=-(q-1)(q^2-1)q^3/3
\] 
if and only if $\chi=\chi_0$, or $\chi=\xi$ where $\xi$ is one of the $(\Gcd(3,q-1)-1)$ non-principal irreducible linear characters of $G$.
\end{enumerate}
\end{lemma}
\begin{proof}
  The character table and a complete information on the conjugacy
  classes of $\PGL_3(q)$ is in~\cite[Section~$3$]{Steinberg}. A direct
  inspection of the tables in~\cite[Section~$3$]{Steinberg} gives that
  $|\mathcal{D}|=(q^2-1)^2q^4/3$ and, using Lemma~\ref{eigenvalues},
  that $\tau=-(q-1)(q^2-1)q^3/3$.  For $q>2$, further inspection of
  the tables in~\cite[Section~$3$]{Steinberg} reveals that
  $\tau=\chi(\mathcal{D})/\chi(1)$ only if $\chi=\chi_0$, of if $\chi$
  is one of the $(\Gcd(3,q-1)-1)$ non-principal irreducible linear
  characters of $G$.
\end{proof}

Applying Lemma~\ref{lemma:7} with the values for $\tau$ and $d$ as in
Lemma~\ref{eigenvaluesPGL3} we obtain the maximal size of an
independent set of $\Gamma_G$.

\begin{corollary}\label{indysetPGL3}
The maximal size of an independent set of $\Gamma_{G}$ is $q^3(q^2-1)(q-1)$.
\end{corollary}
\begin{proof}
  Using the value of $|\mathcal{D}|$ and of the minimum eigenvalue
  $\tau$ of $\Gamma_G$ obtained in Lemma~\ref{eigenvaluesPGL3},
  Lemma~\ref{lemma:7} shows that an independent set of $\Gamma_G$ has
  size no more than $|G|/(1-|\mathcal{D}|/\tau)=q^3(q^2-1)(q-1)$.  The
  stabilizer of a point is an independent set of $\Gamma_G$ of this
  size.
\end{proof}

The same argument can be applied to the group $\PSL_3(q)$.

\begin{corollary}\label{indysetPSL3}
The maximal size of an independent set of $\Gamma_{\PSL_3(q)}$ is $q^3(q^2-1)(q-1)/\Gcd(3,q-1)$.
\end{corollary}
\begin{proof}
  If $\Gcd(3,q-1)=1$, then $\PGL_3(q)=\PSL_3(q)$ and hence from Corollary~\ref{indysetPGL3}
  the maximal size of an independent set of $\Gamma_{\PSL_3(q)}$ is
  $q^3(q^2-1)(q-1)$.

  Suppose that $q-1$ is divisible by $3$. Denote by $\mathcal{D}_0$
  the derangements of $\PSL_3(q)$ and let $\tau_0$ be the minimum
  eigenvalue of $\Gamma_{\PSL_3(q)}$. It can be inferred
  from~\cite[Section~$3$]{Steinberg} that
\[
|\mathcal{D}_0|=(q-1)^2(q+2)(q^2-1)q^3/9 , \qquad \tau_0=-(q-1)^3(q+2)q^2/3.
\]
Since $|\PSL_3(q)|/(1-\frac{|\mathcal{D}_0|}{\tau_0})=q^3(q^2-1)(q-1)/3$,
the rest of the proof follows from Lemma~\ref{lemma:7}.
\end{proof}

Next we will prove Proposition~\ref{rank}.

\begin{proof}[Proof of Proposition~$\ref{rank}$]
For $\alpha,\beta,\gamma,\delta\in \mathcal{P}$, the entry $( A^T A)_{(\alpha,\beta),(\gamma,\delta)}$ equals the number of permutations of
$G$ mapping $\alpha$ into $\beta$ and $\gamma$ into $\delta$. Since
$G$ is $2$-transitive, we get by a simple counting argument that 
\[
(A^T A)_{(\alpha,\beta),(\gamma,\delta)}=
\begin{cases} 
\frac{|G|}{|\mathcal{P}|}=(q-1)(q^2-1)q^3&\textrm{if }\alpha=\gamma \textrm{ and }\beta=\delta,\\
\frac{|G|}{|\mathcal{P}|(|\mathcal{P}|-1)}=(q-1)^2q^2& \textrm{if }\alpha\neq \gamma \textrm{ and }\beta\neq \delta,\\ 
0&\textrm{otherwise.}
\end{cases}
\]
This shows that, with a properly chosen ordering of the columns of $A$,
\[
A^T A=(q-1)(q^2-1)q^3I_{|\mathcal{P}|^2}+(q-1)^2q^2(J_{|\mathcal{P}|}-I_{|\mathcal{P}|})\otimes (J_{|\mathcal{P}|}-I_{|\mathcal{P}|}),
\]
(here $I_n,J_n$ denote the identity matrix and the all-$1$ matrix of size $n$, respectively). 

As the spectrum of the matrix $J_{n}$ is $(n^1, 0^{n-1})$ (the
exponents represent the multiplicities), it follows that the spectrum
of the matrix $(J_n-I_n)\otimes (J_n-I_n)$ is
$(((n-1)^2)^1,1^{(n-1)^2},(-(n-1))^{2(n-1)})$.  Thus $A^T A$ is
diagonalizable with spectrum
\[
\left(  \left( (q+1)(q+2) (q-1)^2 q^3 \right)^{\raisebox{.5em}{$1$}},\quad  
  ((q-1)(q^3-1)q^2)^{\raisebox{.5em}{$(|\mathcal{P}|-1)^2$}},\quad  0^{\raisebox{.5em}{$2(|\mathcal{P}|-1)$}} \right). 
\]
This shows that $A^TA$ has
rank $(|\mathcal{P}|-1)^2+1$ and nullity $2(|\mathcal{P}|-1)$. As $A$ is a matrix with real coefficients, the right kernel of $A$ equals the kernel
of $A^TA$. It follows that $A$ has rank $(|\mathcal{P}|-1)^2+1$ and the dimension of the right kernel of $A$ is
$2(|\mathcal{P}|-1)$.

Finally it is easy to verify that the vectors in $(e_{\bar{\alpha}}^1-e_\alpha^1,e_{\bar{\alpha}}^2-e_\alpha^2\mid \alpha\in\mathcal{P}\setminus\{\bar{\alpha}\})$ are linearly independent and, for each $\alpha\in \mathcal{P}$, we
have
$
A(e_{\bar{\alpha}}^1-e_\alpha^1)=A(e_{\bar{\alpha}}^2-e_\alpha^2)=0.
$
\end{proof}

The only property of $\PGL_3(q)$ that is used in the proof of Proposition~\ref{rank} is the $2$-transitivity. In fact, this result for any $2$-transitive group
is shown in~\cite{AhMe}.

\begin{proof}[Proof of Proposition~$\ref{prop:matrixA}$]
Denote by $\tau$ the minimum eigenvalue of $\Gamma_G$ and write $d=|\mathcal{D}|$.
 If $q=2$, then the proof follows with a computation with the
  computer algebra system \texttt{magma}~\cite{magma}, and hence we 
  assume that $q>2$.

 Let $J$ be the subspace of $\mathbb{C}[G]$ spanned by the characteristic vectors $\chi_S$ of the independent sets $S$ of maximal size of $\Gamma_G$, and let $Z$ be the subspace of $\mathbb{C}[G]$ spanned by the columns of $A$. We have to prove that $Z=J$. To do that we introduce two auxiliary subspaces of $\mathbb{C}[G]$.
 Let $I$ be the ideal of $\mathbb{C}[G]$ generated by the irreducible characters $\chi_0$ and $\one$ (where $\chi_0+\one$ is the permutation character of $G$ acting on $\mathcal{P}$), and let $W$ be the direct sum of
  the $d$-eigenspace and the $\tau$-eigenspace of $\Gamma_G$. 

As each column of $A$ is the characteristic vector of a coset of the stabilizer of a point (and hence the characteristic vector of an independent set of maximal size by Corollary~\ref{indysetPGL3}), the column space $Z$ of $A$ is a subspace of $J$, that is, $Z\leq J$.
Moreover, $\dim Z=(q^2+q)^2+1$ by Proposition~\ref{rank}.

By Lemma~\ref{lemma:7}, if $S$ is an independent set of maximal size of
  $\Gamma_G$, then  $\chi_S\in W$, that is, $J\leq W$. Moreover, $\dim W=(q^2+q)^2+\gcd(3,q-1)$ by Lemma~\ref{eigenvaluesPGL3}~(2).

If $\Gcd(3,q-1)=1$, then $\dim W=\dim Z$ and hence $W=Z=J$.
 Suppose then that $\Gcd(3,q-1)=3$. By
  Lemma~\ref{eigenvaluesPGL3}~(3), $W$ is the ideal of $\mathbb{C}[G]$ generated by $\chi_0$,
  $\one$ and by the two non-principal linear characters of $G$, which we
   call $\xi$ and $\xi^2$. We will show that, for each 
  independent set $S$ of maximal size, we have $\xi(S) =\xi^2(S)=0$. Observe that this implies
  that  $\chi_S$ is orthogonal to $\xi$ and $\xi^2$, and hence $\chi_S$ is contained in the  ideal $I$ of $\mathbb{C}[G]$ generated by
  $\chi_0$ and $\one$, that is, $J\leq I$. Since $\dim I=\chi_0(1)^2+1=(q^2+q)^2+1=\dim Z$, we have $I=Z$ and $J=Z$.

Write $G=\PSL_3(q)\cup \PSL_3(q)x\cup \PSL_3(q)x^2$, for some $x\in
G\setminus\PSL_3(q)$.  Let $S$ be an independent set of maximal size
of $\Gamma_G$ and write $S_{i}=S\cap \PSL_3(q)x^i$, for $i\in
\{0,1,2\}$. Now, $S_ix^{-i}$ is an
independent set of $\Gamma_{\PSL_3(q)}$ and, by
Corollary~\ref{indysetPSL3}, $|S_i|=|S_ix^{-i}|\leq
q^3(q^2-1)(q-1)/3$.

As 
\[
|S_0|+|S_1|+|S_2|=|S|=q^3(q^2-1)(q-1),
\] 
we have $|S_0|=|S_1|=|S_2|$. This shows that the independent set
$S$ is equally distributed among the three cosets of $\PSL_3(q)$ in
$G$. Thus
\[
\xi(S)=|S_0|+(\cos(2\pi/3)+i\sin(2\pi/3))|S_1|+(\cos(4\pi/3)+i\sin(4\pi/3))|S_2|=0
\]
and, similarly,
\[
\xi^2(S)=|S_0|+(\cos(4\pi/3)+i\sin(4\pi/3))|S_1|+(\cos(2\pi/3)+i\sin(2\pi/3))|S_2|=0.
\] 
\end{proof}

\section{Proof of Proposition~\ref{lemma:4}}\label{sec:Proof2}

The real work of this paper is proving Proposition~\ref{lemma:4}. It
is not difficult to show that each of the vectors given in the
statement of Proposition~\ref{lemma:4} is in the right kernel of $M$ but what is
difficult is showing that these vectors actually span the whole right kernel. We will show
that the nullity of $M$ is no more than $4(q^2+q) +1$ by constructing
sufficiently many eigenvectors of the matrix $M^TM$ with non-zero eigenvalues.

Set $N=M^T M$. In particular, $N$ is a square $(q^2+q+1)^2$-matrix whose rows and columns are indexed by the ordered pairs
of points of $\PG_2(q)$, and
\begin{equation}\label{equation:1}
N_{(\alpha,\beta),(\gamma,\delta)}=|\{g\in G\mid \alpha^g=\beta,\,\gamma^g=\delta,\,g\textrm{ derangement}\}|,
\end{equation}
for each $\alpha,\beta,\gamma,\delta\in 	\mathcal{P}$. Since $N$ is symmetric, we have
\begin{equation}\label{equation:2}
N_{(\alpha,\beta),(\gamma,\delta)}=N_{(\gamma,\delta),(\alpha,\beta)}.
\end{equation}
Moreover, it follows at once from~\eqref{equation:1} that
\begin{equation}\label{equation:3}
N_{(\alpha,\beta),(\gamma,\delta)}=N_{(\beta,\alpha),(\delta,\gamma)}.
\end{equation}

The first goal in this section is to calculate the entries of the matrix
$N$ as polynomials in $q$. To do this we need to count the number of monic irreducible
polynomials.

\begin{lemma}\label{monicpoly}
  The number of monic irreducible polynomials of degree $1$, $2$ and
  $3$ over $\mathtt{GF}(q)$ is $q$, $(q-1)q/2$ and $(q^2-1)q/3$,
  respectively.
\end{lemma}
\begin{proof}
  There are $q^{i}$ monic polynomials of degree $i$ with coefficients
  in $\mathtt{GF}(q)$ and, when $i=1$, every such polynomial is
  irreducible over $\mathtt{GF}(q)$.

  Observe that there are $q(q+1)/2$ polynomials of the form
  $(T-x)(T-y)$ with $x,y\in \mathtt{GF}(q)$ (in fact, we have $q$
  polynomials for $x=y$, and $(q-1)q/2$ polynomials for $x\neq
  y$). Therefore the number of monic irreducible polynomials of degree
  $2$ over $\mathtt{GF}(q)$ is $q^2-q(q+1)/2=(q-1)q/2$.

  A similar argument yields the result for polynomials of degree $3$.
\end{proof}

In the following proposition we prove that the entry
$N_{(\alpha,\beta),(\gamma,\delta)}$ of the matrix $N$ is determined
by the geometric position of the $4$-tuple
$(\alpha,\beta,\gamma,\delta)$ in $\PG_2(q)$. First, we define
the following two numbers:
\begin{equation}\label{equation:4}
u=(q-2)q(q^2-1)/3, \qquad v=(q-1)q(q^2-1)/3.
\end{equation}

\begin{proposition}\label{crossratio}Let $\alpha,\beta,\gamma,\delta\in \mathcal{P}$, and let $u$ and $v$ be as in~\eqref{equation:4}. Then
\[
N_{(\alpha,\beta),(\gamma,\delta)}=
\begin{cases}
q^2v&\textrm{if }\alpha\neq\beta, \alpha=\gamma \textrm{ and }\beta=\delta,\\
v&\textrm{if } \alpha,\beta,\gamma,\delta\  \textrm{ are all distinct and}\\
&\textrm{exactly three of }\alpha,\beta,\gamma,\delta \textrm{ are collinear},\\
v&\textrm{if }\alpha=\delta \textrm{ or }\beta=\gamma \textrm{ and not all four of } \alpha,\beta,\gamma,\delta \textrm{ are collinear},\\
u&\textrm{if no three of }\alpha,\beta,\gamma,\delta \textrm{ are collinear},\\
0&\textrm{otherwise.}
\end{cases}
\]
\end{proposition}
\begin{proof}

  Write $n$ for $N_{(\alpha,\beta),(\gamma,\delta)}$. If
  $\alpha=\beta$ or if $\gamma=\delta$, then $n=0$. In particular, we
  may assume that $\alpha\neq \beta$ and $\gamma\neq \delta$.

  If $\alpha=\gamma$ and $\beta=\delta$, then $n$ is the number of
  derangements mapping $\alpha$ to $\beta$. Since $G$ is transitive of
  degree $q^2+q+1$ and since $G$ contains $(q^2-1)^2q^4/3$
  derangements by Lemma~\ref{eigenvaluesPGL3}~(1), we have
  $n=((q^2-1)^2q^4/3)/(q^2+q)=q^2v$. Further, if
  $\alpha=\gamma$ and $\beta\neq \delta$, or $\alpha\neq\gamma$ and
  $\beta=\delta$, then $n=0$ because in both of these cases, there are
  no permutations that map $\alpha$ to $\beta$ and $\gamma$ to
  $\delta$.

  From this point on we assume that $\alpha \neq \gamma$ and $\beta
  \neq \delta$ and we will consider different cases based on the spatial
  arrangement of $\alpha,\beta,\gamma,\delta$.

  First suppose that $\alpha,\beta,\gamma,\delta$ are
  collinear. We are assuming $\alpha\neq \gamma$ and $\beta\neq
  \delta$, and hence $\alpha\vee\gamma=\beta\vee\delta$.  Assume that
  there is a $g\in G$ with $\alpha^g=\beta$ and $\gamma^g=\delta$,
  then
\[
(\alpha\vee \gamma)^g=\alpha^g\vee\gamma^g=\beta\vee \delta=\alpha\vee\gamma.
\]
In particular, $g$ fixes the line $\alpha\vee\gamma$. By the duality
between points and lines of $\PG_2(q)$, we get that $g$ fixes
some point of $\PG_2(q)$ and cannot be a derangement; hence, in this case, $n=0$.

\smallskip

Suppose that exactly three of $\alpha,\beta,\gamma,\delta$ are
collinear.  There are two cases to consider here: first, when all of
$\alpha,\beta,\gamma$ and $\delta$ are distinct and second, when
$|\{\alpha,\beta,\gamma,\delta\}|=3$.  In the second case either $\alpha =
\delta$ or $\beta = \gamma$ (but not both).

Assume that $|\{\alpha,\beta,\gamma,\delta\}|=4$. From the
symmetries in~\eqref{equation:2} and~\eqref{equation:3}, we may assume
that $\alpha,\gamma$ and $\delta$ are collinear. Let $\ell$ be the
line spanned by $\alpha,\gamma,\delta$. Since $G$ is transitive on the
lines in $\mathcal{L}$ and since the stabilizer $G_\ell$ of the line $\ell$
is $2$-transitive on the points in $\ell$ and transitive on the points
in $\mathcal{P} \setminus \ell$, we may assume that
$\ell=\langle\varepsilon_1,\varepsilon_2\rangle$,
$\alpha=\langle\varepsilon_1\rangle$,
$\gamma=\langle\varepsilon_2\rangle$,
$\delta=\langle\varepsilon_1+\varepsilon_2\rangle$ and
$\beta=\langle\varepsilon_3\rangle$. In particular, if $g\in G$ and
$\alpha^g=\beta$, $\gamma^g=\delta$, then
\[g=
\left[
\begin{array}{ccc}
0&\lambda&x\\
0&\lambda&y\\
1&0&z
\end{array}
\right],
\]
for some $\lambda,x,y,z\in \mathtt{GF}(q)$.
The element $g$ is a derangement of $G$ if and only if the characteristic
polynomial $p_{\lambda,x,y,z}(T)$ of $g$ has no root in
$\mathtt{GF}(q)$, and since $p_{\lambda,x,y,z}(T)$ has degree $3$,
this happens exactly when $p_{\lambda,x,y,z}(T)$ is irreducible.
A simple calculation shows that
\[
p_{\lambda,x,y,z}(T)=T^3-(\lambda+z)T^2-(-\lambda z+x)T-\lambda(y-x).
\]
We claim that there are exactly $v$ choices of $(\lambda,x,y,z)$ such
that $p_{\lambda,x,y,z}(T)$ is irreducible. Observe that $\lambda\neq
0$ because $g$ is invertible. It is a
simple computation to check that, for every
$\lambda\in\mathtt{GF}(q)\setminus\{0\}$, and for every $a,b,c\in
\mathtt{GF}(q)$, there exists a unique choice of $x,y,z\in
\mathtt{GF}(q)$ with $$T^3-aT^2-bT-c=p_{\lambda,x,y,z}(T).$$ By
Lemma~\ref{monicpoly}, there are $(q^2-1)q/3$ choices of $a,b,c\in
\mathtt{GF}(q)$ with $T^3-aT^2-bT-c$ irreducible. Therefore, for every
given $\lambda\in\mathtt{GF}(q)\setminus\{0\}$, there exist
$(q^2-1)q/3$ choices for $g$. As we have $q-1$ choices for $\lambda$,
we get $(q-1)q(q^2-1)/3=v$ choices for $g$ in total.

Consider the case when $|\{\alpha,\beta,\gamma,\delta\}|=3$; we have
seen that this implies that either $\alpha=\delta$ or $\beta =
\gamma$.  From the symmetries in~\eqref{equation:2}
and~\eqref{equation:3}, we may assume that $\alpha=\delta$ and that
$\alpha,\beta,\gamma$ are non-collinear (otherwise all four points are
collinear).  Now, replacing $\alpha$,
$\beta$ and $\gamma$ if necessary by $\alpha^h,\beta^h$ and $\gamma^h$
for $h\in G$, we get $\alpha=\langle\varepsilon_1\rangle$,
$\beta=\langle\varepsilon_3\rangle$ and
$\gamma=\langle\varepsilon_2\rangle$. In particular, if $g\in G$ and
$\alpha^g=\beta$, $\gamma^g=\delta$, then
\[g=
\left[
\begin{array}{ccc}
0&\lambda&x\\
0&0&y\\
1&0&z
\end{array}
\right],
\]
for some $\lambda,x,y,z\in \mathtt{GF}(q)$. Arguing as above, the
element $g$ is a derangement of $G$ if and only if the characteristic
polynomial $p_{\lambda,x,y,z}(T)$ of $g$ is irreducible. Now,
\[
p_{\lambda,x,y,z}(T)=T^3-zT^2-xT-\lambda y.
\]
It follows at once from
Lemma~\ref{monicpoly} that there are exactly $v$ choices of
$(\lambda,x,y,z)$ such that $p_{\lambda,x,y,z}(T)$ is irreducible.

\smallskip

Suppose that no three of $\alpha,\beta,\gamma,\delta$ are collinear
(this implies that all four points are distinct). Since $G$ acts
transitively on the ordered $4$-tuples of non-collinear points and
since $N$ is $G$-invariant, we may assume that
$\alpha=\langle\varepsilon_1\rangle$,
$\beta=\langle\varepsilon_3\rangle$,
$\gamma=\langle\varepsilon_2\rangle$ and $\delta=\langle
\varepsilon_1+\varepsilon_2+\varepsilon_3\rangle$. In particular, if
$g\in G$ and $\alpha^g=\beta$, $\gamma^g=\delta$, then
\[g=
\left[
\begin{array}{ccc}
0&\lambda&x\\
0&\lambda&y\\
1&\lambda&z
\end{array}
\right],
\]
for some $\lambda,x,y,z\in \mathtt{GF}(q)$. Arguing as usual, 
the element $g$ is a derangement of $G$ if and only if the characteristic
polynomial
\[
p_{\lambda,x,y,z}(T)=T^3-(\lambda+z)T^2-(\lambda y-\lambda z+x)T-\lambda(y-x)
\]
is irreducible. We claim that there are exactly $u$ choices of
$(\lambda,x,y,z)$ such that $p_{\lambda,x,y,z}(T)$ is
irreducible. Observe that $\lambda\neq 0$ because $g$ is invertible,
and that when $\lambda= -1$ the polynomial $p_{\lambda,x,y,z}(T)$ has root
$-1$. Thus $\lambda\notin\{ 0,-1\}$. It is another simple computation to check
that, for every $\lambda\in\mathtt{GF}(q)\setminus\{0,-1\}$, and for
every $a,b,c\in \mathtt{GF}(q)$, there exists a unique choice of
$x,y,z\in \mathtt{GF}(q)$ with 
\[
T^3-aT^2-bT-c=p_{\lambda,x,y,z}(T).
\]
By Lemma~\ref{monicpoly}, there are $(q^2-1)q/3$ choices of $a,b,c\in
\mathtt{GF}(q)$ with $T^3-aT^2-bT-c$ irreducible. Therefore, for every
given $\lambda\in\mathtt{GF}(q)\setminus\{0,-1\}$, there exist
$(q^2-1)q/3$ choices for $g$. As we have $q-2$ choices for $\lambda$,
we get $(q-2)(q-1)^2q/3=u$ choices for $g$ in total.
\end{proof}

Using Proposition~\ref{crossratio} we deduce a number of properties of
$N$; the first two being that each of $Ne_\alpha^1$, $Ne_\alpha^2$ and $Ne_\ell$ (where
$e_\alpha^1$, $e_\alpha^2$ and $e_\ell$ are defined in~\eqref{eq:ealpha}) are
equal to a multiple of the vector
\begin{equation}\label{eq:defne}
e=\sum_{\substack{\beta,\beta'\in\mathcal{P} \\ \beta\neq \beta'}}e_{\beta\beta'}.
\end{equation}
The $(\alpha,\beta)$-coordinate of $e$ is $1$, unless $\alpha =
\beta$, and in this case the $(\alpha,\beta)$-coordinate is $0$.

\begin{lemma}\label{lemma:Nealpha1}
  For $\alpha\in\mathcal{P}$, we have $Ne_\alpha^1=Ne_\alpha^2=q^2ve$
  (where $v$ is defined in~\eqref{equation:4}).
\end{lemma}
\begin{proof}
  From~\eqref{equation:3}, it suffices to show that
  $Ne_\alpha^1=q^2ve$. From the definition of $e_\alpha^1$, we have
  $Ne_\alpha^1=\sum_{\beta\in\mathcal{P}}Ne_{\alpha\beta}$. In
  particular, by considering the $(\gamma,\delta)$-coordinate of the vector
  $Ne_\alpha^1$, it suffices to prove that
\begin{equation}\label{onemore}
\sum_{\beta\in\mathcal{P}}N_{(\gamma,\delta),(\alpha,\beta)}=
\begin{cases}
q^2v&\textrm{if }\gamma\neq\delta,\\
0&\textrm{if }\gamma=\delta.
\end{cases}
\end{equation}
Clearly, the $(\gamma,\delta)$-row of $N$ is $0$ when $\gamma=\delta$.
Next we consider three cases.

First, suppose that $\alpha=\gamma$. Then Proposition~\ref{crossratio}
shows that the only non-zero summand in the left-hand side of~\eqref{onemore}
occurs for $\beta=\delta$ with value $q^2v$.

Second, suppose that $\alpha\in \gamma\vee\delta$ and $\alpha\neq \gamma$. Now
Proposition~\ref{crossratio} shows that the only non-zero summands in
the left-hand side of~\eqref{onemore} occur when $\beta\notin \gamma\vee\delta$ with value
$v$. Since we have $q^2$ choices for $\beta$, the equality in~\eqref{onemore}
follows.

Finally, suppose that $\alpha\notin \gamma\vee\delta$. Observe that
there are $(q-1)^2$ choices for a point $\beta$ with
$\alpha,\beta,\gamma,\delta$ non-collinear and hence, by
Proposition~\ref{crossratio}, these choices of $\beta$ contribute
$(q-1)^2u$ to the summation in~\eqref{onemore}. Next, we have $3(q-1)$
choices for a point $\beta$ with $|\{\alpha,\beta,\gamma,\delta\}|=4$
and with exactly three of $\alpha,\beta,\gamma,\delta$ collinear (namely,
$q-1$ choices for $\beta$ depending on whether $\alpha,\beta,\gamma$,
or $\alpha,\beta,\delta$, or $\beta,\gamma,\delta$ are
collinear). Each of these terms contributes $v$. In view of
Proposition~\ref{crossratio}, the only remaining choice for $\beta$
that gives a non-zero contribution is when $\beta=\gamma$, and this
single value contributes $v$ to the sum. Thus, in this case the
left-hand side of~\eqref{onemore} equals
\[
(q-1)^2u+3(q-1)v+v=q^2v.
\]
\end{proof}

\begin{lemma}\label{lemma:Neell}
  For $\ell\in\mathcal{L}$, we have $Ne_\ell=q^2ve$ (where $v$ is
  defined in~\eqref{equation:4}).
\end{lemma}
\begin{proof}
  From the definition of $e_\ell$, we have
  $Ne_\ell=\sum_{\alpha,\beta\in\ell}Ne_{\alpha\beta}$. In particular,
  by taking the $(\gamma,\delta)$-coordinate of the vector $Ne_\ell$,
  it suffices to prove that
\[
\qquad \sum_{\alpha,\beta\in \ell}N_{(\gamma,\delta),(\alpha,\beta)}=
\begin{cases}
q^2v&\textrm{if }\gamma\neq \delta,\\
0&\textrm{if }\gamma=\delta.
\end{cases}
\]
Now, the proof follows with a case-by-case analysis (in the same
spirit as the proof of Lemma~\ref{lemma:Nealpha1}) and using
Proposition~\ref{crossratio}.
\end{proof}

\begin{lemma}\label{lemma:-1}
Let $\bar{\alpha}\in\mathcal{P}$ and let $\bar{\ell}\in\mathcal{L}$. Then the subspace
\[
V_0=\langle e_{\alpha\alpha},\,e_{\bar{\alpha}}^1-e_\alpha^1,\,e_{\bar{\alpha}}^2-e_{\alpha}^2,\,e_{\bar{\ell}}-e_\ell
 \mid \alpha\in \mathcal{P} ,\ell\in\mathcal{L}  \rangle
\]
of $V$ has dimension $4(q^2+q)+1$ and is contained in the right kernel of $M$. In particular, $V_0$ is contained in the kernel of $N$.
\end{lemma}
\begin{proof}
  As the $(\alpha,\alpha)$-column of $M$ is zero, the vector
  $e_{\alpha\alpha}$ is clearly in the right kernel of $M$. From
  Lemma~\ref{lemma:Nealpha1}, the vectors
  $e_{\bar{\alpha}}^1-e_\alpha^1$ and
  $e_{\bar{\alpha}}^2-e_{\alpha}^2$ are in the kernel of $N$ and since
  $N=M^TM$, they are also in the right kernel of $M$. Similarly, by
  Lemma~\ref{lemma:Neell}, the vector $e_{\bar{\ell}}-e_\ell$ is in
  the right kernel of $M$.

  Finally we need to confirm that the dimension of $V_0$ is
  $4(q^2+q)+1$. Assume
\[
  \sum_{\alpha \in \mathcal{P}} a_\alpha e_{\alpha\alpha}
+ \sum_{\alpha \in \mathcal{P} \setminus \{\bar{\alpha}\}} b_\alpha \left( e_{\bar{\alpha}}^1-e_\alpha^1\right)
+ \sum_{\alpha \in \mathcal{P} \setminus \{\bar{\alpha}\}} c_\alpha \left( e_{\bar{\alpha}}^2-e_\alpha^2\right)
+ \sum_{\ell\in\mathcal{L}\setminus\{\bar{\ell}\}} d_\ell \left(e_{\bar{\ell}}-e_\ell \right)  = 0,
\]
for some scalars $a_\alpha, b_\alpha, c_\alpha$ and $d_\ell$. Let $\gamma$ and $\delta$ be distinct elements  in $\mathcal{P} \setminus\{\bar{\alpha}\}$ with $\gamma \vee \delta \neq \bar{\ell}$. The $(\gamma, \delta)$-coordinate in the above linear combination yields
\begin{equation}\label{oneonemore}
b_\gamma + c_\delta + d_{\gamma \vee \delta} = 0.
\end{equation}
Observe that if
$\gamma'\in(\gamma\vee\delta)\setminus\{\bar{\alpha},\delta\}$ and
$\delta'\in (\gamma\vee\delta)\setminus\{\bar{\alpha},\gamma\}$, then
(by considering the $(\gamma',\delta)$-coordinate and the
$(\gamma,\delta')$-coordinate) we also get the equations
\[
b_{\gamma'}+c_\delta+d_{\gamma\vee \delta}=0, \qquad 
b_{\gamma}+c_{\delta'}+d_{\gamma\vee \delta}=0.
\] 
Hence $b_\gamma=b_{\gamma'}$ and
$c_{\delta}=c_{\delta'}$. From this, an easy connectedness argument
yields that there exist $b,c\in \mathbb{C}$ with $b=b_\alpha$ and
$c=c_\alpha$, for every
$\alpha\in\mathcal{P}\setminus\{\bar{\alpha}\}$.  This, in turn,
implies that there exists a scalar $d\in\mathbb{C}$ with $d = d_\ell$,
for every $\ell\in \mathcal{L} \setminus \{\bar{\ell}\}$. Note
that~\eqref{oneonemore} gives $b+c+d=0$.

Let $\gamma$ and $\delta$ be distinct elements in $\mathcal{P}\setminus\{\bar{\alpha}\}$ with $\gamma\vee \delta\neq\bar{\ell}$. By considering the $(\bar{\alpha}, \delta)$-coordinate and
the $(\gamma, \bar{\alpha})$-coordinate, we get
\[
(q^2+q) b - c - d = 0 \quad\textrm{and}\quad -b + (q^2+q) c - d = 0.
\]
Putting these two equations together with $b+c+d=0$, we get
$b=c=d=0$. 

Finally, for $\alpha\in \mathcal{P}$, by considering the $(\alpha, \alpha)$-coordinate, we obtain 
$a_\alpha = 0$.
\end{proof}

\begin{lemma}\label{lemma:3}
The vector $e$ is an eigenvector of $N$  with eigenvalue $(q^2+q+1)q^2v$.
\end{lemma}
\begin{proof}
Observe that $e=\sum_{\alpha\in\mathcal{P}}(e_\alpha^1-e_{\alpha\alpha})$ and
that $Ne_{\alpha\alpha}=0$. From Lemma~\ref{lemma:Nealpha1}, we get
\[
Ne=\sum_{\alpha\in \mathcal{P}}Ne_\alpha^1=\sum_{\alpha\in\mathcal{P}}q^2ve=\left(\sum_{\alpha\in
      \mathcal{P}}q^2v\right)e=(q^2+q+1)q^2ve.
\]
\end{proof}

Before exhibiting other eigenvectors of $N$, we need to 
define two  families of vectors, both families are  indexed by
the pairs $(\alpha, \ell)$, where $\alpha$ is a point on the line $\ell$:
\begin{align}\label{eq:ealphaell}
e_{\alpha\ell}=\sum_{\beta\in \ell}e_{\alpha\beta}, \qquad e_{\ell\alpha}=\sum_{\beta\in \ell}e_{\beta\alpha}.
\end{align}
These vectors will be used to construct new  eigenvectors
of $N$, thus our first step is to calculate the value of $Ne_{\alpha \ell}$ and $Ne_{\ell\alpha}$.

\begin{lemma}\label{bastard}
Let $\alpha\in \mathcal{P}$ and let $\ell\in \mathcal{L}$ with $\alpha\in \ell$. Then 
\begin{equation}\label{equation:Nalphaell}
Ne_{\alpha\ell}=q^2v  \sum_{\substack{\beta \in \ell \\ \beta \neq \alpha}} e_{\alpha\beta} 
              +(q-1)v\sum_{\substack{\gamma\not\in\ell \\ \eta \not \in \alpha \vee \gamma }} e_{\gamma \eta} 
              + qv \left( \sum_{\substack{\gamma \not \in \ell \\ \eta \in \alpha \vee \gamma \\ \eta \neq \gamma }}e_{\gamma \eta} 
              + \sum_{\substack{\gamma\in\ell \\ \gamma \neq \alpha \\ \eta \not \in \ell}} e_{\gamma \eta} \right).
\end{equation}
\end{lemma}
\begin{proof}
  Let $\gamma,\delta\in \mathcal{P}$. Denote by $w$ the vector on the
  right-hand side of~\eqref{equation:Nalphaell}. We will compare the
  $(\gamma,\delta)$-coordinate of $Ne_{\alpha\ell}$ and $w$.  First
  note that if $\gamma=\delta$, then
  $(Ne_{\alpha\ell})_{(\gamma,\delta)}=w_{(\gamma,\delta)}=0$, and hence we
  will suppose that $\gamma\neq \delta$. 

The $(\gamma,\delta)$-coordinate of right-hand side of~\eqref{equation:Nalphaell} can be expressed as
\begin{equation}\label{oneoneonemore}
w_{(\gamma,\delta)}=
\begin{cases}
q^2v&\textrm{if }  \ell = \gamma \vee \delta \textrm{ and }\gamma=\alpha,\\
0 & \textrm{if }  \ell = \gamma \vee \delta \textrm{ and }\gamma \neq \alpha,\\
0&\textrm{if }     \ell\neq \gamma\vee\delta \textrm{ and }\gamma=\alpha,\\
(q-1)v&\textrm{if }\ell\neq \gamma\vee\delta, \gamma\neq \alpha \textrm{ and }\ell\wedge(\gamma\vee \delta)\notin\{\alpha,\gamma\},\\
qv&\textrm{if }    \ell\neq \gamma\vee\delta, \gamma\neq \alpha \textrm{ and }\ell\wedge(\gamma\vee \delta)\in \{\alpha,\gamma\}.\\
\end{cases}
\end{equation}

From the definition of $e_{\alpha\ell}$ in~\eqref{eq:ealphaell} and Lemma~\ref{lemma:-1}, we
get $Ne_{\alpha\ell}=\sum_{\beta\in \ell\setminus\{\alpha\}}Ne_{\alpha\beta}$. In particular,
$(Ne_{\alpha\ell})_{(\gamma,\delta)}=\sum_{\beta\in\ell\setminus\{\alpha\}}N_{(\gamma,\delta),(\alpha,\beta)}$. We
now compare this number with~\eqref{oneoneonemore}. 

First assume that $\ell=\gamma\vee \delta$.  If $\alpha=\gamma$,
then from Proposition~\ref{crossratio} we see that
$N_{(\gamma,\delta),(\alpha,\beta)}\neq 0$ only if $\beta=\delta$ and
in this case its value is $q^2v$. If
$\alpha\neq\gamma$, then  Proposition~\ref{crossratio} gives
$\sum_{\beta\in \ell\setminus\{\alpha\}}N_{(\gamma,\delta),(\alpha,\beta)}=0$.

Now assume that $\ell\neq \gamma\vee\delta$ and $\gamma=\alpha$. As $\ell\neq \gamma\vee\delta$, we get 
$\delta \not \in \ell$, and hence $\delta \neq \beta$, for all $\beta
\in \ell\setminus\{\alpha\}$. Thus $\sum_{\beta\in
  \ell\setminus\{\alpha\}}N_{(\gamma,\delta),(\alpha,\beta)}=0$. 

Now assume that $\ell\neq \gamma\vee\delta$, $\gamma\neq\alpha$ and $\ell\wedge(\gamma\vee\delta)\notin\{\alpha,\gamma\}$. Here  we  consider two cases depending on whether $\ell\wedge(\gamma\vee\delta) = \delta$ or $\ell\wedge(\gamma\vee\delta)\neq \delta$. If $\ell\wedge(\gamma\vee\delta)=\delta$, then
Proposition~\ref{crossratio} gives that
$N_{(\gamma,\delta),(\alpha,\beta)}=v$ for every $\beta\in
\ell\setminus\{\alpha,\delta\}$, and
$N_{(\gamma,\delta),(\alpha,\beta)}=0$ when $\beta=
\delta$. Thus $\sum_{\beta\in
  \ell\setminus\{\alpha\}}N_{(\gamma,\delta),(\alpha,\beta)}=(q-1)v$.
Now suppose that $\ell\wedge(\gamma\vee\delta)\neq\delta$, and hence $\ell\wedge(\gamma\vee\delta)\notin\{\alpha,\gamma,\delta\}$. Define $\alpha'=\ell\wedge(\gamma\vee\delta)$. Now, Proposition~\ref{crossratio} gives
that $N_{(\gamma,\delta),(\alpha,\beta)}=u$ for every $\beta\in\ell\setminus\{\alpha,\alpha'\}$, and $N_{(\gamma,\delta),(\alpha,\beta)}=v$ when $\beta=\alpha'$. Thus
$\sum_{\beta\in \ell\setminus\{\alpha\}}N_{(\gamma,\delta),(\alpha,\beta)}=(q-1)u+v=(q-1)v$.

Finally, assume that $\ell\neq \gamma\vee\delta$, $\gamma\neq
\alpha$ and $\ell\wedge(\gamma\vee\delta)\in \{\alpha,\gamma\}$. Here Proposition~\ref{crossratio} gives  $N_{(\gamma,\delta),(\alpha,\beta)}=v$, for every $\beta \in\ell\setminus\{\alpha\}$. Thus $\sum_{\beta\in \ell\setminus\{\alpha\}}N_{(\gamma,\delta),(\alpha,\beta)}=qv$. 
\end{proof}

\begin{remark}\label{remark}{\rm In the proof of Lemma~\ref{eigenvalueeasy} we will use~\eqref{equation:Nalphaell}. However, for the computations there it is  convenient to express the equality in~\eqref{equation:Nalphaell} as a linear combination of vectors of the form $e_{\alpha'\ell'}$. It is straightforward to see  that
\begin{align*}
Ne_{\alpha\ell}&=q^2v(e_{\alpha\ell}-e_{\alpha\alpha})+(q-1)v\sum_{\gamma\notin\ell}(e_\gamma^1-e_{\gamma(\alpha\vee\gamma)})\\\nonumber
&+
qv\left(\sum_{\gamma\notin \ell}(e_{\gamma(\alpha\vee\gamma)}-e_{\gamma\gamma})+\sum_{\gamma\in\ell\setminus\{\alpha\}}(e_{\gamma}^1-e_{\gamma\ell})\right).\\\nonumber
\end{align*}}
\end{remark}

We define two new sets of vectors of $V$ that we will show are eigenvectors of $N$.
Given three non-collinear points $\alpha, \beta, \gamma$ we write
\begin{equation}\label{equation:10}
\begin{aligned}
e_{\alpha\beta\gamma}^1&=
(e_{\alpha(\alpha\vee\beta)}-e_{\beta(\alpha\vee\beta)})+
(e_{\beta(\beta\vee\gamma)}-e_{\gamma(\beta\vee\gamma)})+
(e_{\gamma(\gamma\vee\alpha)}-e_{\alpha(\gamma\vee\alpha)}),\\
e_{\alpha\beta\gamma}^2&=
(e_{(\alpha\vee\beta)\alpha}-e_{(\alpha\vee\beta)\beta})+
(e_{(\beta\vee\gamma)\beta}-e_{(\beta\vee\gamma)\gamma})+
(e_{(\gamma\vee\alpha)\gamma}-e_{(\gamma\vee\alpha)\alpha}).
\end{aligned}
\end{equation}

\begin{lemma}\label{eigenvalueeasy}
Let $\alpha\in \mathcal{P}$ and let  $\ell\in \mathcal{L}$ with $\alpha\in \ell$. If
  $\alpha,\beta,\gamma$ are non-collinear, then
  $e_{\alpha\beta\gamma}^1$ and $e_{\alpha\beta\gamma}^2$ are
  eigenvectors of $N$ with eigenvalue $(q^2+q+1)v$. 
Moreover, the subspace
$
\langle  e_{\alpha\beta\gamma}^1,e_{\alpha\beta\gamma}^2\mid  \alpha,\beta,\gamma \textrm{ non-collinear points}\rangle
$ 
of $V$ has dimension $2q^3$.
\end{lemma}
\begin{proof}
  Let $\alpha,\beta,\gamma$ be non-collinear points. We start by  showing that
  $Ne_{\alpha\beta\gamma}^1=e_{\alpha\beta\gamma}^1(q^2+q+1)v$. Observe that the same equality holds for  $e_{\alpha\beta\gamma}^2$ from~\eqref{equation:3}. 

From~\eqref{equation:Nalphaell} and Remark~\ref{remark} applied with
  $\ell=\alpha\vee\beta$, we obtain that
\begin{align*}
N(e_{\alpha(\alpha\vee\beta)}-e_{\beta(\alpha\vee\beta)})=&
q^2v( e_{\alpha (\alpha \vee \beta)} - e_{\beta (\alpha \vee \beta)}) -q^2v (e_{\alpha\alpha} - e_{\beta \beta}) \\
&+ (q-1)v \left( \sum_{\zeta \not \in \alpha \vee \beta }e_{\zeta (\zeta \vee \beta)}-e_{\zeta  (\zeta \vee \alpha)}\right) \\
&+ qv \left( \sum_{\zeta \not \in \alpha \vee \beta}  e_{\zeta (\zeta \vee \alpha)} - e_{\zeta (\zeta \vee \beta)} \right) 
   +qv \left( \sum_{\eta \not \in \alpha \vee \beta} e_{\beta \eta} - e_{\alpha \eta}\right) \\
=&
q^2v \left( e_{\alpha (\alpha \vee \beta)} - e_{\beta (\alpha \vee \beta)} \right)  -q^2v (e_{\alpha\alpha} - e_{\beta \beta}) \\
& + v \left( \sum_{\zeta \not \in \alpha \vee \beta}  e_{\zeta (\zeta \vee \alpha)} - e_{\zeta (\zeta \vee \beta)} \right) 
+qv \left( \sum_{\eta \not \in \alpha \vee \beta} e_{\beta \eta} - e_{\alpha \eta}\right).
\end{align*}

Applying this formula to each of
$e_{\alpha(\alpha\vee\beta)}-e_{\beta(\alpha\vee\beta)}$,
$e_{\beta(\beta\vee\gamma)}-e_{\gamma(\beta\vee\gamma)}$ and
$e_{\gamma(\alpha\vee\gamma)}-e_{\alpha(\alpha\vee\gamma)}$, we get
that $Ne_{\alpha\beta\gamma}^1$ equals
\begin{align*}
& q^2v \left(  e_{\alpha (\alpha \vee \beta)} - e_{\beta (\alpha \vee \beta)} + e_{\beta (\beta \vee \gamma)} - e_{\gamma (\beta \vee \gamma)}+ e_{\gamma (\gamma \vee \alpha) } - e_{\alpha (\gamma \vee \alpha)}\right) \\
&+ v \left( \sum_{\zeta \not \in \alpha \vee \beta} ( e_{\zeta (\zeta \vee \alpha)} - e_{\zeta (\zeta \vee \beta)}) + \sum_{\zeta \not \in \beta \vee \gamma}  (e_{\zeta (\zeta \vee \beta)} - e_{\zeta (\zeta \vee \gamma)})+ \sum_{\zeta \not \in \gamma \vee \alpha} (e_{\zeta (\zeta \vee \gamma)} - e_{\zeta (\zeta \vee \alpha)}) \right)\\
&+q v \left( \sum_{\eta \not \in \alpha \vee \beta} e_{\beta \eta} - e_{\alpha \eta} + \sum_{\eta \not \in \beta \vee \gamma} e_{\gamma \eta} - e_{\beta \eta} + \sum_{\eta \not \in \gamma \vee \alpha} e_{\alpha \eta} - e_{\gamma \eta}\right). 
\end{align*}
Write $Ne_{\alpha\beta\gamma}^1=q^2vX_1+vX_2+qvX_3$, where $q^2vX_1,vX_2,qvX_3$ are the three summands in the above equation.
From~\eqref{equation:10}, we have $X_1=e_{\alpha\beta\gamma}^1$, we will show that $X_2$ and $X_3$ equal $e_{\alpha\beta\gamma}^1$ as
well.

Note that by rearranging the terms,  $X_2$ can be rewritten as
\begin{align}\label{eqeqeqeq1}
 \! \!\sum_{\zeta \notin \alpha \vee \beta}\! e_{\zeta \zeta\vee \alpha}\!-\! \!\sum_{\zeta \notin  \gamma \vee \alpha}\! e_{\zeta \zeta\vee \alpha} 
   + \! \!\sum_{\zeta \notin \beta \vee\gamma}\! e_{\zeta \zeta\vee \beta} \!-\!\! \sum_{\zeta \notin \alpha \vee \beta }\! e_{\zeta \zeta\vee \beta}
   + \!\! \sum_{\zeta \notin  \gamma \vee \alpha }\! e_{\zeta \zeta\vee \gamma}\!-\!\!\sum_{\zeta \notin \beta \vee \gamma}\! e_{\zeta \zeta\vee \gamma}. 
\end{align}

Observe now that
\begin{equation*}
\sum_{\zeta \notin \alpha \vee \beta}\! e_{\zeta (\zeta\vee \alpha)}\!-\! \!\sum_{\zeta \notin  \gamma \vee \alpha}\! e_{\zeta (\zeta\vee \alpha)}=\sum_{\substack{\zeta \in \gamma \vee \alpha \\ \zeta \neq \alpha}} e_{\zeta (\zeta\vee \alpha)} 
     \!-\!\!\sum_{\substack{\zeta \in \alpha \vee \beta \\ \zeta \neq \alpha}} e_{\zeta (\zeta\vee \alpha)},
\end{equation*}
and that a similar equality holds for the other two terms in~\eqref{eqeqeqeq1}. Therefore $X_2$ equals
\begin{align*}
&  \!\sum_{\substack{\zeta \in \gamma \vee \alpha \\ \zeta \neq \alpha}} e_{\zeta \zeta\vee \alpha} 
     \!-\!\!\sum_{\substack{\zeta \in \alpha \vee \beta \\ \zeta \neq \alpha}} e_{\zeta \zeta\vee \alpha}
     +\!\sum_{\substack{\zeta \in \alpha \vee \beta \\ \zeta \neq \beta}} e_{\zeta \zeta\vee \beta}
     \!-\!\!\sum_{\substack{\zeta \in \beta \vee \gamma \\ \zeta \neq \beta}} e_{\zeta \zeta\vee \beta}
     +\!\sum_{\substack{\zeta \in \beta \vee \gamma \\\zeta \neq \gamma}} e_{\zeta \zeta\vee \gamma}
     \!-\!\!\sum_{\substack{\zeta \in \gamma \vee \alpha  \\ \zeta \neq \gamma}} e_{\zeta \zeta\vee \gamma}\\
 =&  (\!\sum_{\substack{\zeta \in \alpha \vee \beta \\ \zeta \neq \beta}}\! e_{\zeta \zeta\vee \beta} 
              \!-\!\! \sum_{\substack{\zeta \in \alpha \vee \beta \\ \zeta \neq \alpha}}\! e_{\zeta \zeta\vee \alpha})
     +(\!\sum_{\substack{\zeta \in \alpha \vee \gamma \\ \zeta \neq \alpha}}\! e_{\zeta \zeta\vee \alpha} 
              \!-\!\!\sum_{\substack{\zeta \in \alpha \vee \gamma \\ \zeta \neq \gamma}}\! e_{\zeta \zeta\vee \gamma})
     +(\!\sum_{\substack{\zeta \in \beta \vee \gamma\\\zeta \neq \gamma}}\! e_{\zeta \zeta\vee \gamma} 
            \! - \!\!\sum_{\substack{\zeta \in \beta \vee \gamma \\ \zeta \neq \beta}}\! e_{\zeta \zeta\vee \beta} ) \\
 =& (e_{\alpha (\alpha \vee \beta)} - e_{\beta (\alpha \vee \beta)})
   + (e_{\beta (\beta \vee \gamma)} - e_{\gamma (\beta \vee \gamma)})
   + (e_{\gamma (\gamma \vee \alpha) } - e_{\alpha (\gamma \vee \alpha)})=e_{\alpha\beta\gamma}^1.
\end{align*}

By rearranging the terms, we see that $X_3$ equals
\begin{align*}
 \left( \sum_{\eta \not \in \gamma \vee \alpha} e_{\alpha \eta}  - \sum_{\eta \not \in \alpha \vee \beta}e_{\alpha \eta} \right)
     + \left( \sum_{\eta \not \in \alpha \vee \beta} e_{\beta \eta} - \sum_{\eta \not \in \beta \vee \gamma} e_{\beta \eta} \right)
     + \left( \sum_{\eta \not \in \beta \vee \gamma} e_{\gamma \eta} -  \sum_{\eta \not \in \gamma \vee \alpha}  e_{\gamma \eta} \right), 
\end{align*}
which reduces to $(e_{\alpha \alpha \vee \beta} - e_{\alpha \alpha \vee \gamma}) 
+ (e_{\beta \beta \vee \gamma} 
- e_{\beta \beta \vee \alpha}) 
+ (e_{\gamma \gamma \vee \alpha} 
- e_{\gamma \gamma \vee \beta})=e_{\alpha\beta\alpha}^1$.
Putting all of these together, we have that $Ne_{\alpha\beta\gamma}^1 = (q^2+q+1)v e_{\alpha\beta\gamma}^1$.

It remains to show that these vectors span a subspace of dimension
$2q^3$. Fix $\bar{\ell}$ and $\bar{\alpha}\in\bar{\ell}$.  We will
show that the vectors in
\[
T= \{ e_{\bar{\alpha}\beta\gamma}^1, e_{\bar{\alpha}\beta\gamma}^2 \mid 
 \beta\in\mathcal{P}\setminus\bar{\ell},
\gamma\in\bar{\ell}\setminus \{\bar{\alpha}\} \}
\]
are linearly independent.  Assume that
\begin{align}\label{linearcombO}
 \sum_{\substack{\beta\in\mathcal{P}\setminus\bar{\ell} \\ \gamma\in\bar{\ell}\setminus \{\bar{\alpha}\}}} a_{\beta\gamma} \; e_{\bar{\alpha}\beta\gamma}^1 
+ \sum_{\substack{\beta\in\mathcal{P}\setminus\bar{\ell} \\ \gamma\in\bar{\ell}\setminus \{\bar{\alpha}\}}} b_{\beta\gamma} \; e_{\bar{\alpha}\beta\gamma}^2 = 0.
\end{align}

Pick $\zeta$ and $\eta$ in $\mathcal{P}\setminus \bar{\ell}$ with $(\zeta \vee
\eta) \cap \bar{\ell} \neq \bar{\alpha}$. Then the $(\zeta,
\eta)$-coordinate of $e_{\bar{\alpha}\beta\gamma}^1$ is non-zero only if
$\beta = \zeta$ and $\gamma = (\zeta \vee \eta) \cap \bar{\ell}$.
Similarly, the $(\zeta, \eta)$-coordinate of $e_{\bar{\alpha}\beta\gamma}^2$
is non-zero only if $\beta = \eta$ and $\gamma = (\zeta \vee \eta) \cap
\bar{\ell}$. Thus, for any three collinear points $\zeta, \eta, \gamma$ with $\gamma
\in \bar{\ell}$, we have
\[
a_{\zeta \gamma} + b_{\eta \gamma}=0.
\]
Observe that if $\zeta'\in (\zeta\vee\eta)\setminus\{\eta\}$ and $\eta'\in (\zeta\vee\eta)\setminus\{\zeta\}$, then by applying the same argument we obtain $a_{\zeta'\gamma}+b_{\eta\gamma}=0$ and $a_{\zeta\gamma}+b_{\eta'\gamma}=0$. Hence, with a connectedness argument, we get both
\[
a_{\zeta \gamma} = a_{\eta \gamma}, \qquad b_{\zeta\gamma}=b_{\eta\gamma}.
\]

Next consider $\gamma \in \bar{\ell}$ and $\eta \notin \bar{\ell}$. By taking  the $(\eta,
\gamma)$-coordinate of~\eqref{linearcombO} we get
\[
a_{\eta\gamma} + \sum_{\substack{\beta \in \eta \vee \gamma \\ \beta \neq \gamma}} b_{\beta\gamma}=0.
\]
By varying $\eta$ with other points on the line $\eta \vee \gamma$ we
get 
\[
a_{\eta\gamma}=0, \qquad \sum_{\substack{\beta \in \eta \vee \gamma \\ \beta \neq \gamma}} b_{\beta \gamma} = 0.
\]
Thus $a_{\eta\gamma}=b_{\eta\gamma}=0$, for all $\eta \notin
\bar{\ell}$ and $\gamma \in \ell \setminus \{\alpha\}$.
\end{proof}

We give one last family of eigenvectors of $N$, these are
based on four collinear points.  Given four distinct collinear points
$\alpha,\beta,\gamma,\delta$, define
\begin{equation}\label{equation:10-1}
e_{\alpha\beta\gamma\delta}=(e_{\alpha\gamma}-e_{\alpha\delta})+(e_{\beta\delta}-e_{\beta\gamma}).
\end{equation}

\begin{lemma}\label{eigenvaluehard}
  If $\alpha,\beta,\gamma,\delta$ are four distinct collinear points,
  then $e_{\alpha\beta\gamma\delta}$ is an eigenvector of $N$ with
  eigenvalue $q^2v$. Moreover, for $q>2$, the subspace 
\[
\langle  e_{\alpha\beta\gamma,\delta}\mid \alpha,\beta,\gamma,\delta \textrm{ distinct collinear points}\rangle
\]
of $V$ has dimension at least $(q^2+q+1)(q^2-q-1)$.
\end{lemma}

\begin{proof}
  Let $\zeta,\eta\in \mathcal{P}$ and  let $\ell$ be the line containing
  $\alpha,\beta,\gamma,\delta$. From Proposition~\ref{crossratio}, we
  see with a case-by-case analysis that
  $(Ne_{\alpha\gamma})_{(\zeta,\eta)}\neq
  (Ne_{\alpha\delta})_{(\zeta,\eta)}$ only if one of the following
  occurs: 
\begin{enumerate}[(i)]
\item $\zeta=\alpha$ and $ \eta = \gamma$;
\item $\zeta=\alpha$ and $ \eta = \delta$;
\item $\zeta=\delta$ and $\nu\notin\ell$;
\item $\zeta=\gamma$ and $\nu\notin\ell$; 
\item $\zeta,\eta\notin \ell$, $\zeta \neq \eta$ and $\{\zeta,\eta,\gamma\}$ are collinear; 
\item $\zeta,\eta\notin \ell$, $\zeta \neq \eta$ and $\{\zeta,\eta,\delta\}$ are collinear.
\end{enumerate}
Now, by considering these six cases separately, we get
\[
(N(e_{\alpha\gamma}-e_{\alpha\delta}))_{(\zeta,\eta)}=
\begin{cases}
q^2v&\textrm{if } \zeta=\alpha \textrm{ and } \eta=\gamma,\\
-q^2v&\textrm{if } \zeta=\alpha \textrm{ and } \eta=\delta,\\
v&\textrm{if } \zeta=\delta \textrm{ and } \eta\notin\ell,\\
-v&\textrm{if } \zeta=\gamma \textrm{ and } \eta\notin\ell,\\
v-u&\textrm{if } \zeta,\eta \notin\ell, \zeta \in\gamma\vee \eta \textrm{  and }\zeta \neq \eta\\
-(v-u)&\textrm{if }\zeta,\eta \notin\ell,  \zeta \in\delta\vee \eta \textrm{ and  }\zeta \neq \eta,\\
0&\textrm{otherwise}.
\end{cases}
\]
This shows that
\begin{align*}
N(e_{\alpha\gamma}-e_{\alpha\delta})&=q^2v(e_{\alpha\gamma}-e_{\alpha\delta})+v\left((e_\delta^2-e_{\ell\delta})-(e_\gamma^2-e_{\ell\gamma})\right) \\
& +(v-u)\left( \sum_{\nu\in\mathcal{P}\setminus\ell}(e_{\nu(\gamma\vee\nu)}-e_{\nu\nu}-e_{\gamma\gamma})
     -\sum_{\nu\in\mathcal{P}\setminus \ell}(e_{\nu(\delta\vee\nu)}-e_{\nu\nu}-e_{\delta\delta}) \right).
\end{align*}
An immediate application of this formula to $N(e_{\alpha\gamma}-e_{\alpha\delta})$ and $N(e_{\beta\delta}-e_{\beta\gamma})$ gives
\[
Ne_{\alpha\beta\gamma\delta}=q^2ve_{\alpha\beta\gamma\delta},
\]
and $e_{\alpha\beta\gamma\delta}$ is an eigenvector of $N$ with eigenvalue $q^2v$.

For $q>2$, it remains to give a lower bound on the dimension
of 
\[
W=\langle e_{\alpha\beta\gamma\delta}\mid \alpha,\beta,\gamma,\delta
\textrm{ distinct collinear points}\rangle.
\]
For $\ell\in\mathcal{L}$, set 
$
W_\ell=\langle e_{\alpha\beta\gamma\delta}\mid \alpha,\beta,\gamma,\delta \textrm{ distinct
  points in }\ell\rangle.
$
From the definition of $W$, we have
$W=\sum_{\ell\in\mathcal{L}}W_\ell$. Moreover, for $\ell,\ell'\in
\mathcal{L}$ with $\ell\neq \ell'$, we see that $W_\ell$ is orthogonal
(with respect to the standard scalar product) to $W_{\ell'}$ because no
two distinct points can lie in $\ell\wedge \ell'$. Thus
$W=\oplus_{\ell\in \mathcal{L}}W_\ell$. Since $|\mathcal{L}|=q^2+q+1$,
it remains to prove that $\dim W_\ell\geq q^2-q-1$.

Fix $\ell\in \mathcal{L}$ and write
$\ell=\{\alpha_0,\ldots,\alpha_q\}$. Consider the family of vectors
\begin{align*}
\mathcal{F}'=&\{ e_{\alpha_i\alpha_{q-1}\alpha_{j}\alpha_q} \mid 0\leq
i,j\leq q-2,\,i\neq j\}\cup
\{ e_{\alpha_{i}\alpha_{q-2}\alpha_{q-1}\alpha_q} \mid 0\leq i\leq q-3\}\\
&\cup\{e_{\alpha_{q-1}\alpha_q\alpha_i\alpha_{q-2}}\mid 0\leq i\leq q-3\}
\end{align*}
and set
$\mathcal{F}=\mathcal{F}'\cup\{e_{\alpha_{q-2}\alpha_q\alpha_{q-3}\alpha_{q-1}}\}$.
Clearly, $|\mathcal{F}|=(q-1)(q-2)+(q-2)+(q-2)+1=q^2-q-1$. We claim
that the vectors in $\mathcal{F}$ are linearly independent.  To see
this order the elements $(e_{\alpha_{i}\alpha_{j}})_{0\leq i,j\leq q}$
with the lexicographic order, that is,
$e_{\alpha_i\alpha_j}<e_{\alpha_{i'}\alpha_{j'}}$ if $i<i'$, or $i=i'$
and $j<j'$. By writing each $e_{\alpha\beta\gamma\delta}$ as a
$\{0,1\}$-vector of length $(q+1)^2$ with respect to the basis
$(e_{\alpha_i\alpha_j})_{0\leq i,j\leq q}$ and with respect to this
ordering, we see that the elements in $\mathcal{F}'$ are in echelon
form. In particular, the elements in $\mathcal{F}'$ are linearly
independent. Finally, it is easy to see that
$e_{\alpha_{q-2}\alpha_q\alpha_{q-3}\alpha_{q-1}}$ is linearly
independent with the elements of $\mathcal{F}'$ as the $(q,q-1)$-entry
is non-zero in this vector, but it is equal to zero in every vector in
$\mathcal{F}'$.
\end{proof}

\begin{proof}[Proof of Proposition~$\ref{lemma:4}$]
  
  If $q=2$, then the proof follows with a computation with the
  computer algebra system \texttt{magma}~\cite{magma}, so we
  assume that $q>2$.

Let $V_0$ be the subspace defined in Lemma~\ref{lemma:-1}.  From
Lemmas~\ref{lemma:3},~\ref{eigenvalueeasy} and~\ref{eigenvaluehard},
the rank of $N$ is at least $1+2q^3+(q^2+q+1)(q^2-q-1)$ and hence the
kernel of $N$ has dimension no more than
$(q^2+q+1)^2-(1+2q^3+(q^2+q+1)(q^2-q-1))=4(q^2+q)+1$. Therefore
Lemma~\ref{lemma:-1} gives that the kernel of $N$ has dimension
exactly $4(q^2+q)+1$ and hence equals $V_0$. As $N=M^TM$, we get that
the right kernel $M$ is also $V_0$.
\end{proof}

\end{document}